\documentclass{amsart}

\usepackage{amsmath, amssymb, amsthm}
\usepackage{enumerate}
\usepackage[margin=1in]{geometry}
\usepackage{xcolor}

\title{Energy-minimizing measures supported near fractal $1$-sets}
\author{Rosemarie Bongers}
\address{Department of Applied Mathematics \\ University of California, Merced}
\email{rosemariebongers@ucmerced.edu}
\date{\today}

\newtheorem{theorem}{Theorem}[section]
\newtheorem{defn}[theorem]{Definition}
\newtheorem{lemma}[theorem]{Lemma}
\newtheorem{corollary}[theorem]{Corollary}

\newtheorem{conjecture}[theorem]{Conjecture}

\newtheorem*{theorem*}{Theorem}
\newtheorem*{problem*}{Problem}
\newtheorem*{corollary*}{Corollary}

\newcommand{\dyad}[2]{\mathcal{D}_{#1, #2}}
\newcommand{\dyadsum}[3]{\sum_{#1 \in \dyad{#2}{#3}}}
\newcommand{\dyadsubsum}[4]{\sum_{\substack{#1 \in \dyad{#2}{#3} \\ #1 \subseteq #4}}}

\begin{document}
\maketitle

\begin{abstract}
    Energy techniques can be used to study the structure of fractal sets; the existence of a measure with finite Riesz energy supported on a set gives information about its dimension,  distribution, and density. In this paper, we study energy-minimizing measures supported near fractal $1$-sets. Using physical analogy and a variant of the fast multipole method, we show a strong equidistribution result for these measures. We impose only mild geometric constraints on our sets, assuming only a generational structure of the approximations. This allows us to consider sets which do not exhibit self-similarity or other algebraic constraints. As a corollary, we demonstrate a fundamental limitation in the use of energy techniques for studying Favard length.
\end{abstract}

\tableofcontents

\section{Introduction}
Given a positive Borel measure $\mu$ supported on a subset of $\mathbb{R}^n$ and a parameter $s > 0$, the \emph{Riesz $s$-energy} of $\mu$ is defined by
$$I_s(\mu) = \iint \frac{d\mu(x) \, d\mu(y)}{|x - y|^s}.$$
The finiteness of this quantity gives important information about the geometric distribution of the support of $\mu$; if $\mu$ concentrates on a set of Hausdorff dimension strictly less than $s$, the singularity $|x - y|^{-s}$ is not integrable, and the corresponding energy is infinite. This motivates the definition of the \emph{Riesz $s$-capacity} of a set:
$$\operatorname{Cap}_s(E) = \sup I_s(\mu)^{-1}$$
where the supremum is taken over positive Borel probability measures with support contained in $E$. The Riesz $s$-capacity of a set is positive if and only if there exists a positive measure supported on $E$ with finite $s$-energy; morally, this implies that the set must have dimension at least $s$, since any smaller-dimensional set would require too much concentration of the measure on low-dimensional sets. Frostman's lemma makes this connection precise:

\begin{theorem*}[\protect{\cite[Theorem 8.9]{Mat95}}]
If $A$ is a Borel set in $\mathbb{R}^n$ and $s > 0$, then
\begin{itemize}
    \item If $\mathcal{H}^s(A) < \infty$ then $\operatorname{Cap}_s(A) = 0.$
    \item If $\operatorname{Cap}_s(A) = 0$, then $\mathcal{H}^t(A) = 0$ for all $t > s$.
\end{itemize}
\end{theorem*}
See, e.g. \cite[Chapter 8]{Mat95} for a more complete exposition of the connections between Riesz capacity and Hausdorff measures.

The energy of a measure carries far more information than just the dimension of its support; it also reflects the geometric distribution of the underlying set. Roughly speaking, a highly dispersed a set is (that is, having low $s$-dimensional density) will support measures with lower energies than a more concentrated set. This suggests a connection with more geometrically-motivated objects. To this end, the \emph{Favard length} of a set in the plane is defined by
$$\operatorname{Fav}(E) = \int_{S^{1}} |\pi_{\theta} (E)| \, d\theta,$$
where $| \cdot |$ is the one-dimensional Hausdorff measure (or Lebesgue measure) and $\pi_{\theta}$ denotes orthogonal projection into a line through the origin at angle $\theta$ from the $x$-axis. A foundational result of Mattila \cite[Theorem 3.2]{Mat90} is that the Favard length is connected with $1$-energy:
\begin{theorem*} If $E \subset \mathbb{C}$ is a compact set which supports a positive Borel probability measure $\mu$,
$$\operatorname{Fav}(E) \gtrsim I_1(\mu)^{-1}.$$
\end{theorem*}
Thus if one is searching for lower bounds on the Favard length of a set, then one may ask an alternative question: find the measure supported on the set which extremizes the $1$-energy, and calculate its energy. This leads us to the following:

\begin{problem*} If $E \subseteq \mathbb{C}$ is a compact set, find the Borel probability measure $\mu$ supported on $E$ with minimal $1$-energy.
\end{problem*}

This problem is very physically motivated, and such a measure is referred to as the equilibrium measure. One can imagine placing a distribution of point charges within $E$, in which case the Riesz energy is a measurement of the total repulsion of the configuration. It is natural to expect that such charges will distribute in a fairly uniform manner, avoiding concentration as much as possible. The \emph{poppy seed bagel} theorem of Hardin and Saff \cite{HarSaf05} proves this result for rectifiable manifolds in $\mathbb{R}^d$; see also their expository article \cite{HarSaf04}. A more recent development of Hardin, Reznikov, Saff, and Volberg \cite{HarRezSafVol19} studied the local properties and regularity of energy-minimizing measures and proved a separation condition for optimal point configurations given a certain degree of regularity of the underlying set. Related work of Anderson, Reznikov, Vlasiuk, and White \cite{AndRezVlaWhi22} considers applications of potential theory to polarization and covering problems.

The fractal case is frequently more challenging, since their intricate geometric structure can preclude the kind of even distribution that is observed for rectifiable manifolds. Results in this direction frequently consider the optimal packing of point configurations using tools from potential theory; see, e.g. the work of Anderson and Reznikov \cite{AndRez24} for the asymptotics for minimal $s$-energy of $N$-point configurations contained within a class of self-similar sets of non-integer dimension. For a computational perspective that also applies to fractals, Rajon, Ransford, and Rostand \cite{RajRanRos10} phrased the energy-minimization problem in terms of a quadratic form and developed an algorithm which can be applied to a large class of sets; their results yield numerical estimates of capacities (including logarithmic, Riesz, and hyperbolic) for various specific sets in $\mathbb{R}^2$ and $\mathbb{R}^3$. 

In this paper, the sets we will consider are those which approximate $1$-dimensional fractals. The prototypical example we will have in mind is the $n$-th generation $\mathcal{K}_n$ of the four-corner (or Garnett) Cantor set, as it is a typical testbed for conjectures in geometric measure theory. The sets are defined recursively; $\mathcal{K}_0$ is the unit square $[0, 1]^2$. The next generation is $\mathcal{K}_1$, which is defined via slicing $\mathcal{K}_0$ into sixteen squares (four equal cuts in the vertical and horizontal directions) and taking the four squares at the corners.
Subsequent sets are defined by taking each constituent square in $\mathcal{K}_n$ and repeating the four-corner construction. This yields a sequence of sets $\mathcal{K}_n$ each consisting of $4^n$ squares with sidelength $4^{-n}$. The limit of this process is $\mathcal{K} = \bigcap_{n = 0}^{\infty} \mathcal{K}_n$, which is a purely unrectifiable set with positive and finite 1-Hausdorff measure. Since this is a purely unrectifiable 1-dimensional set, the Besicovitch projection theorem implies that
$$0 = \operatorname{Fav}(\mathcal{K}) = \lim_{n \to \infty} \operatorname{Fav}(\mathcal{K}_n),$$
so one may study the asymptotic rate of decay of the Favard length of the generations. 

A challenging open question in geometric measure theory is to understand the sharp asymptotics of $\operatorname{Fav}(\mathcal{K}_n)$. Energy techniques can prove that
$$\operatorname{Fav}(\mathcal{K}_n) \gtrsim \frac 1 n.$$
This follows from the aforementioned result of Mattila; one may directly construct a measure supported on $\mathcal{K}_n$ which has energy comparable to $n$. To wit, the equidistributed measure on $\mathcal{K}_n$ achieves this -- that is, the measure which assigns mass $4^{-n}$ uniformly to each of the $4^n$ constituent squares. The current best result is due to Bateman and Volberg \cite{BatVol10} and improves this by a logarithmic factor: 
$$\operatorname{Fav}(\mathcal{K}_n) \gtrsim \frac{\log n}{n}.$$
This result does not use energy techniques, but rather relies on a delicate square-counting argument that varies based on the projection angle. One may ask whether energy techniques could be used to prove a better result than $\operatorname{Fav}(\mathcal{K}_n) \gtrsim 1/n$; this would be equivalent to constructing a measure supported on $\mathcal{K}_n$ with energy substantially lower than $n$. The goal of this paper is to prove that this is not possible -- any probability measure supported on $\mathcal{K}_n$ will have energy $\gtrsim n$.

The main result of this paper is to show an equidistribution result for energy minimizers on a broad class of sets in the plane -- roughly speaking, a fractal poppy seed bagel theorem. Our approach is inspired by the fast multipole method, which is used to simplify computational algorithms for the dynamics of $N$-body gravitational systems. This algorithm groups approximates complex systems of close-together masses as single point masses in order to reduce the computational expense of modeling a large number of pairwise interactions. Our approach here is to replace the Riesz potential with a notion of discretized \emph{repulsion}, which approximates away the fine structure of a fractal set. Under a mild geometric constraint (which is closely related to being an approximant of a fractal $1$-set), the repulsion of a measure controls the Riesz energy from below. We then show that repulsion-minimizing measures are equidistributed in a precise sense using a \emph{mass-exchange} principle: if a measure concentrates its mass at any location, we can exchange its mass to a less-concentrated region and decrease the overall repulsion. 

The goal of this paper is two-fold: we wish to give an ``elementary'' characterization of energy-minimizers without using many tools of potential theory, and we wish to give as general a geometric result as possible. Our results will apply to any set which exhibits a \emph{generational} structure, where we can find finer and finer approximations to the set by taking a sequence of cubes, each having a unique ancestor and finitely many children. We now begin to make these notions precise. 

\subsection{Notation, definitions, and preliminaries}
We begin with our key geometric assumption: that the set under consideration has a generational structure. 

\begin{defn} Given a natural number $n$,  we say that a set $E$ has the \emph{filtration property through generation $n$} if for each integer $0 \le \ell \le n$, there is a pairwise disjoint family of sets $\dyad E \ell$ for which 
\begin{itemize}
\item If $0 \le k \le \ell$ is an integer and $R \in \dyad E \ell$, there is a unique $P \in \dyad E k$ with $R \subseteq P.$ In this case we call $P$ the \emph{generation $k$ ancestor} of $R$ and write $\widetilde{R}^{(k)}$ for $P$.
\item $\displaystyle E = \bigcup_{R \in \dyad E n} R.$
\end{itemize}
\end{defn}

In the case that $k = \ell - 1$, we simply write $\widetilde{Q}$ for its (immediate) ancestor. The notation is meant to evoke the idea of a dyadic grid at scale $2^{-\ell}$, where each generation $\ell$ cube has children at subsequent generations and a unique ancestor at each preceding generation.  The prototypical example for a set with the filtration property (through any generation) is the $n$-th generation $\mathcal{K}_n$ of the four-corner Cantor set.

\begin{defn} Given a natural number $n$, a set $E$ with the filtration property through generation $n$, and an increasing sequence $\{r_{\ell} : 0 \le \ell \le n\}$, the \emph{repulsion} between two sets $Q, R \in \dyad E n$ is
$$r(Q, R) = r_{\ell},$$
where $\ell$ is the largest integer for which $\widetilde{Q}^{\ell} = \widetilde{R}^{\ell}.$ If $\mu$ is a measure supported on $E$, its \emph{repulsion at generation $n$} is defined as
$$\mathcal{Q}(\mu) = \dyadsum Q E n \dyadsum R E n \mu(Q) \mu(R)r(Q,R) .$$
In a case where it is important to highlight the sequence $\vec r = \{r_{\ell} : 0 \le \ell \le n\}$, we explicitly notate the corresponding repulsion as $\mathcal{Q}_{\vec r}(\mu).$
\end{defn}

We will be able to characterize repulsion-minimizing measures using only the filtration property. In order to translate this to \emph{energy}-minimizers, we will use one additional property; qualitatively speaking, the set should exhibit the same degree of concentration (i.e. number of children per ancestor) at all locations.

\begin{defn}
Suppose that $n$ is a natural number and $E$ is a set with the filtration property through generation $n$. We say that $E$ has the \emph{socialist filtration property} if there is a sequence of nonnegative integers $N_{\ell}$ such that for every $P \in \dyad E \ell$,
$$\left|\left\{Q \in \dyad E {\ell + 1} : Q \subseteq P\right\}\right| = N_{\ell}.$$
\end{defn}

\subsection{Main results and outline}
The main result of this paper is the following.
\begin{theorem*}[\ref{thm:socialist}]
Fix a set $E$ and a natural number $n$ for which $E$ has the socialist filtration property through generation $n$, as well as a sequence $\{r_{\ell} : 0 \le \ell \le n\}$. Let $\mu$ be a positive Borel probability measure supported on $E$ which minimizes the repulsion induced by $\mathcal{Q}$. If $A, B \in \dyad E n$, we have $\mu(A) = \mu(B).$
\end{theorem*}

Because the generations $\mathcal{K}_n$ satisfy the socialist filtration property, we also have the following corollary:
\begin{corollary*}[3.4]
If $\mu$ is a positive Borel probability measure supported on $\mathcal{K}_n$, then
$$I_1(\mu) \gtrsim n.$$
\end{corollary*}
This demonstrates a fundamental limitation in the application of energies to Favard length: they cannot be used to show any result sharper than $\operatorname{Fav}(\mathcal{K}_n) \gtrsim \frac 1n.$

The remainder of the paper is organized as follows:
\begin{itemize}
    \item In Section 2, we show that repulsion-minimizing measures on sets with a filtration property are equidistributed. This proceeds in a few steps:
    \begin{itemize}
        \item We show that repulsion-minimizing measures are non-degenerate, in that they cannot assign zero mass to any constituent of the set.
        \item Mass-exchange then allows us to prove a local equidistribution result at the finest scale of the filtration.
        \item Under the socialist filtration assumption, the local equidistribution result can be boostrapped to a global one.
    \end{itemize}
    \item In Section 3, we characterize energy-minimizing measures supported on $\mathcal{K}_n$ and discuss
    sharpness.
    \item In Section 4, we discuss a general computational approach following the lines of \cite{RajRanRos10}, and phrase a linear algebraic conjecture which would lead to 
    sharp asymptotics for a broad class of fractal sets.
\end{itemize}

\section{Repulsion-minimizing measures on sets with a filtration}

\subsection{Non-degeneracy of repulsion minimizers} \label{RepulsionStep1}
\begin{lemma}
Suppose that $\mu$ is a positive Borel probability measure supported on a set $E$ which satisfies the filtration property through generation $n$. If $\mathcal{Q}(\mu)$ is minimal over all such measures, then $$\mu(A) > 0 \quad\quad \text{for all $A \in \dyad E n$.}$$ 
\end{lemma}

\begin{proof}
As a warmup, suppose that there are siblings $A, B \in \dyad E n$ (that is, $\tilde{A} = \tilde{B}$) but $\mu(A) = 0$ and $ \mu(B) > 0$. Define a new measure $\nu$ such that $\nu(A) = \nu(B) = \frac 1 2 \mu(B)$ with $\left.\nu\right|_{(A \cup B)^c} = \left.\mu\right|_{(A \cup B)^c}.$ This is still a Borel probability measure, and we can calculate the difference in repulsions as follows:
\begin{align*}
\Delta \mathcal{Q} := \mathcal{Q}(\mu) - \mathcal{Q}(\nu) = \dyadsum Q E n \dyadsum R E n r(Q, R) \left[\mu(Q) \mu(R) - \nu(Q) \nu(R)\right].
\end{align*}
Our goal is to show that this overall sum is positive; this will show that $\mu$ is non-optimal. We now handle three cases based on $Q$:

\begin{itemize}
\item If $Q = A$, the relevant part of the sum is 
\begin{align*}
\dyadsum R E n &r(A, R) \left[\mu(Q) \mu(R) - \nu(Q) \nu(R)\right] \\
&= r(A, A) \left[\mu(A)^2 - \nu(A)^2\right] + r(A, B) \left[\mu(A) \mu(B) - \nu(A) \nu(B)\right] \\
&\quad + \sum_{R \ne A, B} r(A, R) \left[\mu(A) \mu(R) - \nu(A) \nu(R)\right] \\
&= - \frac 1 4 r_n \mu(B)^2 - \frac 1 4 r_{n - 1} \mu(B)^2 - \frac 1 2 \mu(B) \sum_{R \ne A, B} r(A, R) \mu(R).
\end{align*}
The last step follows as $A$ and $B$ are siblings.
\item If $Q = B$, the relevant part of the sum is
\begin{align*}
\dyadsum R E n &r(B, R) \left[\mu(B) \mu(R) - \nu(B) \nu(R)\right] \\
&= r(B, A) \left[\mu(B) \mu(A) - \nu(B) \nu(A) \right] + r(B, B) \left[\mu(B)^2 - \nu(B)^2\right] \\
&\quad + \sum_{R \ne A, B} r(B, R) \left[\mu(B) \mu(R) - \nu(B) \nu(R)\right] \\
&= -\frac 1 4 r_{n - 1} \mu(B)^2 + \frac 3 4 r_n \mu(B)^2 + \frac 1 2 \mu(B) \sum_{R \ne A, B} r(B, R) \mu(R).
\end{align*}
\item If $Q \ne A$ or $B$, the relevant part of the sum is
\begin{align*}
\dyadsum R E n r(Q, R) \left[\mu(Q) \mu(R) - \nu(Q) \nu(R)\right] &= \mu(Q) \dyadsum R E n r(Q, R) \left[\mu(R) - \nu(R)\right] \\
&= \mu(Q) \big(r(Q, A) \left[\mu(A) - \nu(A)\right] + r(Q, B) \left[\mu(B) - \nu(B)\right]\big) \\
&= \mu(Q) \left(-\frac 1 2 r(Q, A) \mu(B) + \frac 1 2 r(Q, B) \mu(B)\right) \\
\end{align*}
after noting that $\mu(R) = \nu(R)$ unless $R = A$ or $R = B$. Furthermore, because $A$ and $B$ share a parent and $Q$ is not equal to either $A$ or $B$, we have $r(Q, A) = r(Q, B)$. Therefore, this term is zero for all such $Q$. 
\end{itemize}
Combining the three cases, we have
\begin{align*}
\Delta \mathcal{Q} &= \frac 1 2 r_n \mu(B)^2 - \frac 1 2 r_{n - 1} \mu(B)^2 + \frac 1 2 \mu(B) \sum_{R \ne A, B} r(B, R) \mu(R) - r(A, R) \mu(R).
\end{align*}
Again using the fact that $r(B, R) = r(A, R)$ for all $R \ne A, B$, we are left with
$$\Delta \mathcal{Q} = \frac{1}{2} \left(r_n - r_{n - 1}\right) \mu(B)^2  > 0$$
as desired.

In the general case, suppose that there is an $A \in \dyad E n$ with $\mu(A) = 0$. The idea of the proof will again be to transfer mass from a region with positive mass; however, there is not a natural way to do this in general -- especially for a set which is not self-similar. The key idea will be that we can regard the overall repulsion as the integral of a \emph{potential} determined by other constituents:
$$\mathcal{Q}(\mu) = \dyadsum Q E n \mu(Q) \underbrace{\dyadsum R E n r(Q, R) \mu(R)}_{\text{potential}}.$$
A region of $E$ which has zero mass will have a strictly smaller potential than any nearby region of $E$ with positive mass. A result of this is that transferring mass from a nearby constituent will decrease the repulsion of the measure, again showing that $\mu$ is not optimal. We now make this precise.

Fix a constituent $A \in \dyad E n$ with $\mu(A) = 0$. Since $\mu$ is a non-zero measure, there is a generation $k$ such that the ancestor $\widetilde{A}^{(k)}$ has non-zero mass but $\mu(\widetilde{A}^{(k + 1)}) = 0.$ Choose any constituent $B \in \dyad E n$ with $B \subseteq \widetilde{A}^{(k)}$ and $\mu(B) > 0$. Define $\nu$ as in the previous case:
$$\nu(A) = \nu(B) = \frac 1 2 \mu(B), \quad \left.\nu\right|_{(A \cup B)^c} = \left.\mu\right|_{(A \cup B)^c}.$$
We then have
$$\Delta \mathcal Q := \dyadsum Q E n \dyadsum R E n r(Q, R) \left[\mu(Q) \mu(R) - \nu(Q) \nu(R)\right]$$
and proceed in cases based on the geometry of $Q$ and $R$. For notation, let $N = E \setminus (\widetilde{A}^{(k + 1)} \cup \widetilde{B}^{(k + 1)})$; this corresponds to the part of $E$ which is neither close to $A$ nor $B$. By an abuse of notation,
$$\Delta \mathcal Q = \left(\dyadsubsum Q E n {\widetilde{A}^{(k + 1)}} + \dyadsubsum Q E n {\widetilde{B}^{(k + 1)}} + \dyadsubsum Q E n N\right)\left(\dyadsubsum R E n {\widetilde{A}^{(k + 1)}} + \dyadsubsum R E n {\widetilde{B}^{(k + 1)}} + \dyadsubsum R E n N\right) r(Q, R) \left[\mu(Q) \mu(R) - \nu(Q) \nu(R)\right].$$
This suggests nine cases based on the relative geometry; this is immediately reduced to six cases by observing the relevant symmetries.

\begin{enumerate}[(1)]
\item $Q, R \subseteq \widetilde{A}^{(k+1)}$,
\item $Q\subseteq \widetilde{A}^{(k+1)}, R\subset \widetilde{B}^{(k+1)}$ and its symmetric version with $Q,R$ interchanged,
\item $Q \subset \widetilde{A}^{(k+1)}, R\subseteq N$ and its symmetric version,
\item $Q, R \subseteq \widetilde{B}^{(k+1)}$,
\item $Q \subseteq \widetilde{B}^{(k+1)}, R\subset N$ and its symmetric version, and
\item $Q,R\subseteq N$.
\end{enumerate}

\textbf{Case 1:} Observing that $\mu(\widetilde{A}^{(k + 1)} \setminus A) = \nu(\widetilde{A}^{(k + 1)} \setminus A) = 0$, this reduces to the case $Q = R = A$ and yields $$\Delta \mathcal{Q}_{\text{Case 1}} = -\frac 1 4 r_n \mu(B)^2.$$

\textbf{Case 2:} If $Q \subseteq \widetilde{A}^{(k + 1)}$ and $R \subseteq \widetilde{B}^{(k + 1)},$ as in Case 1, the only surviving terms correspond to $Q = A$. Note that since $Q$ and $R$ have the same ancestor at generation $k$ but not $k + 1$, we have $r(Q, R) = r_{k}$. This yields
\[
\dyadsubsum R E n {\widetilde{B}^{(k + 1)}} r_{k} \left[0 - \nu(Q) \nu(R)\right] \\
= -\frac 1 2 r_{k} \mu(B) \nu(\widetilde{B}^{(k + 1)}).
\]
Here we use that $\nu(A) = \frac{1}{2}\mu(B)$. The symmetric case is identical, yielding 
\[
\Delta \mathcal{Q}_{\text{Case 2}} = - r_k \mu(B) \mu(B) \nu(\widetilde{B}^{(k+1)}).
\]

\textbf{Case 3:} If $Q \subseteq \widetilde{A}^{(k + 1)}$ and $R \subseteq N$, $\mu(R) = \nu(R)$ and terms with $Q \ne A$ vanish. This leaves
$$ -\frac 1 2 \mu(B) \dyadsubsum R E n N r(A, R)\mu(R).$$ 
The symmetric case is identical, yielding 
\[
\Delta \mathcal{Q}_{\text{Case 3}} = - \mu(B) \dyadsubsum R E n N r(A, R)\mu(R).
\]


\textbf{Case 4:} $Q, R \subseteq \widetilde{B}^{(k + 1)}$. The key observation of for this case is that whenever $S \in \dyad E n \cap (A \cup B)^c$ we have $\nu(S) = \mu(S)$. If $Q = R = B$, we have
$$\mu(Q) \mu(R) - \nu(Q) \nu(R) = \mu(B)^2 - \nu(B)^2 = \frac 3 4 \mu(B)^2.$$
Otherwise, write
$$\mu(Q) \mu(R) - \nu(Q) \nu(R) = \mu(Q) \big(\mu(R) - \nu(R)\big) + \nu(R) \big(\mu(Q) - \nu(Q)\big).$$
If $Q = B$ but $R \ne B$, the first term vanishes and this equates to
$$\frac 1 2 \mu(B) \mu(R).$$
The case that $Q \ne B$ and $R = B$ is handled identically. 
Finally, if $Q \ne B$ and $R \ne B$, both terms vanish. Combining these (and observing the symmetry between the cases $Q = B, R \ne B$ and $Q \ne B, R = B$) we are left with
\begin{align*}
\Delta \mathcal{Q}_{\text{Case 4}} &= \frac 3 4 \mu(B)^2 r_n + \mu(B) \dyadsubsum R E n {\widetilde{B}^{(k + 1)} \setminus B} r(B, R) \mu(R).
\end{align*}

\textbf{Case 5:} If $Q \subseteq \widetilde{B}^{(k + 1)}$ and $R \subseteq N$, as in Case 3, we have $\mu(R) = \nu(R)$. Furthermore, $\mu(Q) = \nu(Q)$ for all $Q \ne B$, so we are left with
\begin{align*}
\dyadsubsum Q E n {\widetilde{B}^{(k + 1)}} \dyadsubsum R E n N r(Q, R) \mu(R) \left(\mu(Q) - \nu(Q)\right) = \frac 1 2 \mu(B) \dyadsubsum R E n N r(B, R) \mu(R).
\end{align*}
The symmetric version is identical, yielding
\[
\Delta \mathcal{Q}_{\text{Case 5}} = \mu(B) \dyadsubsum R E n N r(B, R) \mu(R)
\]



\textbf{Case 6:} If $Q, R \subseteq N$, $\mu(Q) \mu(R) = \nu(Q) \nu(R)$ and there is zero contribution.

We are now ready to combine all the cases:
\begin{align*}
\Delta \mathcal{Q} &= \underbrace{-\frac 1 4 r_n \mu(B)^2}_{\text{Case 1}} - \underbrace{r_{k} \mu(B) \nu(\widetilde{B}^{(k + 1)})}_{\text{Case 2}} \\
&\quad - \underbrace{\mu(B) \dyadsubsum R E n N r(A, R) \mu(R)}_{\text{Case 3}} + \underbrace{\frac 3 4 \mu(B)^2 r_n + \dyadsubsum R E n {\widetilde{B}^{(k + 1)} \setminus B} r(B, R) \mu(R)}_{\text{Case 4}} \\
&+ \underbrace{\mu(B) \dyadsubsum R E n N r(B, R) \mu(R)}_{\text{Case 5}} \\
&= \frac 1 2 r_n \mu(B)^2 - r_{k} \mu(B) \nu(\widetilde{B}^{(k + 1)}) + \dyadsubsum R E n N \big(r(B, R) - r(A, R)\big) \mu(R) \\
&\quad + \mu(B) \dyadsubsum R E n {\widetilde{B}^{(k + 1)} \setminus B} r(B, R) \mu(R).
\end{align*}
For any $R \subseteq N$, we have that $r(B, R) = r(A, R)$ and so the third term is zero. The second term can be rewritten using
$$\nu(\widetilde{B}^{(k + 1)}) = \nu(\widetilde{B}^{(k + 1)} \setminus B) + \nu(B) = \mu(\widetilde{B}^{(k + 1)} \setminus B) + \frac 1 2 \mu(B).$$
Therefore,
$$\Delta \mathcal{Q} = \frac 1 2 r_n \mu(B)^2 - r_{k} \mu(B) \left(\frac 1 2 \mu(B) + \mu(\widetilde{B}^{(k + 1)} \setminus B)\right) + \mu(B) \dyadsubsum R E n {\widetilde{B}^{(k + 1)} \setminus B} r(B, R) \mu(R).$$
As a check on this formula, it is worth comparing it to the case that $A$ and $B$ are siblings. In that case we would take $k + 1 = n$; the third term degenerates and the first two reduce to $\frac{1}{2} \left(r_n - r_{n - 1}\right) \mu(B)^2$ as before.

To complete the proof, we need to show that $\Delta \mathcal{Q} > 0$. The challenge is the negative middle term; this corresponds to a loss in repulsion by redistributing mass away from $B$. This is compensated for in the final sum. To this end, note that if $R \subseteq \widetilde{B}^{(k + 1)} \setminus B$ there is a common ancestor between $R$ and $B$ at a generation which is at least $k + 1$; that is, $r(B, R) \ge r_{k + 1}$. As a consequence,
\begin{align*}
\dyadsubsum R E n {\widetilde{B}^{(k + 1)} \setminus B} r(B, R) \mu(R) &\ge r_{k + 1} \dyadsubsum R E n {\widetilde{B}^{(k + 1)} \setminus B} \mu(R) = r_{k + 1} \mu(B^{k + 1} \setminus B).
\end{align*}
Finally, we have
\begin{align*}
\Delta \mathcal{Q} &\ge \frac 1 2 r_n \mu(B)^2 - \frac 1 2 r_k \mu(B)^2 - r_k \mu(B) \mu(\widetilde{B}^{(k + 1)} \setminus B) + r_{k + 1} \mu(B) \mu(\widetilde{B}^{(k + 1)} \setminus B) \\
&= \frac{1}{2} (r_n - r_k) \mu(B)^2 + (r_{k + 1} - r_k) \mu(B) \mu(\widetilde{B}^{(k + 1)} \setminus B) .
\end{align*}
Since $r_n > r_{k + 1} > r_k$ and $\mu(B) > 0$, we have $\Delta \mathcal{Q} > 0$ as desired. 
\end{proof}

\subsection{Local equidistribution of repulsion minimizers} \label{RepulsionStep2}

\begin{defn}[Mass-exchanged measure]
Given sets $A$ and $B$ and a measure $\mu$ for which $\mu(A) > 0$ and $\mu(B) > 0$, the \emph{mass-exchanged measure} $\tilde \mu$ is the measure for which
\begin{itemize}
\item $\left.\tilde \mu\right|_{(A \cup B)^c} = \left.\mu\right|_{(A \cup B)^c}$
\item $\left.\tilde \mu\right|_{A} = \frac{\mu(A) + \mu(B)}{2 \mu(A)} \left.\mu\right|_A,$
\item and $\left.\tilde \mu\right|_{B} = \frac{\mu(A) + \mu(B)}{2 \mu(B)} \left.\mu\right|_B$.
\end{itemize}
\end{defn}
In other words, $\tilde \mu$ is the measure which is locally scaled on $A$ and $B$ to ensure that they have equal mass, while remaining unmodified on the remainder of the support of $\mu$. As a notational comment, the exchange $\tilde \mu$ implicitly depends on both $A$ and $B$; when the mass-exchanged sets are clear from context, we do not explicitly notate this.
\begin{lemma}[Mass-exchange] Fix natural numbers with $0 \le k \le n$ and suppose that $E$ has the filtration property through generation $n$. Furthermore, fix an increasing sequence $\{r_{\ell} : 0 \le \ell \le n\}$.

Suppose that $\mu$ is a positive Borel probability measure supported on $E$ and that $A, B \in \dyad E {k+1}$ with $\tilde A = \tilde B$ but $A \ne B$, and with $\mu(A) > 0$ and $\mu(B) > 0$. If $\tilde{\mu}$ is the corresponding mass-exchanged measure,
\begin{align*}
\mathcal{Q}(\mu) - \mathcal{Q}(\tilde \mu) &= \sum_{k \le \ell \le n - 1} (r_{\ell + 1} - r_{\ell}) \left(\dyadsubsum P E {\ell + 1} A \left[1 - \left(\frac{\mu(A) + \mu(B)}{2\mu(A)}\right)^2 \right] \mu(P)^2\right. \\
&\hspace{.2in} \left.+ \dyadsubsum P E {\ell + 1} B \left[1 - \left(\frac{\mu(A) + \mu(B)}{2\mu(B)}\right)^2 \right] \mu(P)^2\right)
\end{align*}
where $\mathcal{Q}$ is the repulsion implicitly depending on the sequence $\{r_{\ell} : 0 \le \ell \le n\}.$
\end{lemma}

\begin{proof}
We will compute $\mathcal{Q}(\mu) - \mathcal{Q}(\tilde \mu)$ and show that it is nonnegative; to this end, we have
\begin{align*}
\Delta \mathcal Q = \mathcal Q(\mu) - \mathcal Q(\tilde \mu) &= \dyadsum Q E n \dyadsum R E n \big[\mu(Q) \mu(R) - \tilde \mu(Q) \tilde \mu(R)\big] r(Q, R).
\end{align*}

 Note that the mass-exchanged measure $\tilde \mu$ is unmodified on $E \setminus (A \cup B)$; we now consider the different conditions that are possible for squares $Q, R \in \dyad E n$. There are a total of nine different cases, based on inclusion in $A$, $B$, and $(A \cup B)^c$; this is immediately reduced to six cases by observing the relevant symmetries:
\begin{enumerate}[(1)]
\item $Q \subseteq A$, $R \subseteq A$,
\item $Q \subseteq B$, $R \subseteq B$,
\item $Q \subseteq A$, $R \subseteq B$ and its symmetric version with $Q$ and $R$ interchanged,
\item $Q \subseteq A$, $R \nsubseteq (A \cup B)$ and its symmetric version,
\item $Q \subseteq B$, $R \nsubseteq (A \cup B)$ and its symmetric version, and
\item $Q \nsubseteq (A \cup B)$, $R \nsubseteq (A \cup B).$
\end{enumerate}

{\textbf{Case 1:}} If $Q$ and $R$ are both subsets of $A$, there is a unique last common ancestor $P \in \dyad E {\ell}$ for some index $k \le \ell \le n$. That is, we have $\widetilde{Q}^{(\ell)} = \widetilde{R}^{(\ell)} = P$ but $\widetilde{Q}^{(\ell + 1)} \ne \widetilde{R}^{(\ell + 1)}$. (We address the case when $Q = R$ momentarily.) In this case, we have
\begin{align*}
\left[\mu(Q) \mu(R) - \tilde \mu(Q) \tilde \mu(R)\right] r(Q, R) &= \left[\mu(Q) \mu(R) - \left(\frac{\mu(A) + \mu(B)}{2 \mu(A)} \mu(Q)\right) \cdot \left(\frac{\mu(A) + \mu(B)}{2 \mu(A)} \mu(R) \right)\right] r_{\ell} \\ 
&= \left[1 - \left(\frac{\mu(A) + \mu(B)}{2 \mu(A)}\right)^2\right] \mu(Q) \mu(R) r_{\ell}.
\end{align*}
Now fix a $Q$ and a scale $\ell$; in this case,
\begin{align*}
\sum_{\substack{R \in \dyad E n \\ R \subseteq A \\ r(Q, R) = r_{\ell}}} \left[\mu(Q) \mu(R) - \tilde \mu(Q) \tilde \mu(R)\right] r(Q, R) &= \left[1 - \left(\frac{\mu(A) + \mu(B)}{2\mu(A)}\right)^2 \right] \mu(Q) r_{\ell} \sum_{\substack{R \in \dyad E n \\ R \subseteq A \\ r(Q, R) = r_{\ell}}} \mu(R).
\end{align*}
Note that when $\ell = k$, the fact that $A \in \dyad E {k + 1}$ implies that the sum is empty and therefore contributes nothing. For terms with $\ell \ge k + 1$, the condition $R \subseteq A$ is now redundant. Because $Q$ and $R$ have a common ancestor $P \in \dyad E {\ell}$ but are not contained in the same child of $P$, the lattermost sum is
\begin{align*}
\sum_{\substack{R \in \dyad E n \\ r(Q, R) = r_{\ell}}} \mu(R) &= \sum_{\substack{S \in \operatorname{ch}(P) \\ S \ne \widetilde{Q}^{\ell + 1}}} \sum_{\substack{R \in \dyad E n \\ R \subseteq S}} \mu(R) \\
&= \sum_{\substack{S \in \operatorname{ch}(P) \\ S \ne \widetilde{Q}^{\ell + 1}}} \mu(S) \\
&= \mu(P) - \mu(\widetilde{Q}^{(\ell + 1)})
\end{align*}
where we have denoted $\operatorname{ch}(P)$ as the immediate children of $S$. Therefore, once we have our $Q$ and $\ell$ fixed, and recalling that $P = \widetilde{Q}^{\ell}$, we have
$$\quad \sum_{\substack{R \in \dyad E n \\ R \subseteq A \\ r(Q, R) = r_{\ell}}} \left[\mu(Q) \mu(R) - \tilde \mu(Q) \tilde \mu(R)\right] r(Q, R) = \left[1 - \left(\frac{\mu(A) + \mu(B)}{2\mu(A)}\right)^2 \right] \mu(Q) r_{\ell} \cdot \big(\mu(\widetilde{Q}^{(\ell)}) - \mu(\widetilde{Q}^{(\ell + 1)})\big).$$
In the case that $Q = R$, the computation is much simpler because the sum over $R$ degenerates into the single term $\mu(Q)$. As we interpret $\tilde{Q}^{(n + 1)} = \emptyset$, the above equation remains valid even with $\ell = n$. Again with $Q$ fixed, we may now apply summation by parts over scales $k + 1 \le \ell \le n$:
\begin{align*}
\sum_{k+1 \le \ell \le n} r_{\ell} \cdot &\big(\mu(\widetilde{Q}^{(\ell)}) - \mu(\widetilde{Q}^{(\ell + 1)})\big) \\
&= \sum_{k+1 \le \ell \le n} (-r_{\ell}) \cdot (\mu(\widetilde{Q}^{(\ell + 1)}) - \mu(\widetilde{Q}^{(\ell)})) \\
&= (-r_{n + 1}) \mu(\widetilde{Q}^{(n + 1)}) - (-r_k)(\mu(\widetilde{Q}^{(k)}) - \sum_{k \le \ell \le n} \mu(\widetilde{Q}^{(\ell + 1)}) \cdot ((-r_{\ell + 1}) - (-r_{\ell})) \\
&= r_k \mu(\widetilde{Q}^{(k)}) + \sum_{k \le \ell \le n} \mu(\widetilde{Q}^{(\ell + 1)}) (r_{\ell + 1} - r_{\ell}).
\end{align*}
Observe that we have an ascending chain
$$\emptyset = \widetilde{Q}^{(n + 1)} \subset \widetilde{Q}^{(n)} \subset \cdots \subset \widetilde{Q}^{(k)} = A$$
induced by $Q \in \dyad E n$, ordered by decreasing generation (and thus increasing scale). Returning to the original goal of this case (that is, to compute the change in $\mathcal{Q}$ as we exchange mass within the measure), we have
\begin{align*}
\Delta_{\text{Case 1}} \mathcal{Q} &= \dyadsubsum Q E n A \dyadsubsum R E n A \left[\mu(Q) \mu(R) - \tilde \mu(Q) \tilde \mu(R)\right] r(Q, R) \\
&= \left[1 - \left(\frac{\mu(A) + \mu(B)}{2\mu(A)}\right)^2\right] \dyadsubsum Q E n A \mu(Q) \sum_{k \le \ell \le n} \sum_{\substack{R \in \dyad E n \\ R \subseteq A \\ r(Q, R) = r_{\ell}}} \mu(R) \cdot r_{\ell} \\
&= \left[1 - \left(\frac{\mu(A) + \mu(B)}{2\mu(A)}\right)^2\right] \dyadsubsum Q E n A \mu(Q) \left[r_k \mu(\widetilde{Q}^{(k)}) + \sum_{k \le \ell \le n} \mu(\widetilde{Q}^{(\ell + 1)}) (r_{\ell + 1} - r_{\ell})\right].
\end{align*}
For the first component of this sum, we may use the fact that $\widetilde{Q}^{(k)} = A$ to calculate
$$\dyadsubsum Q E n A \mu(Q) \cdot r_k \cdot \mu(\widetilde{Q}^{(k)}) = \mu(A) \cdot r_k \cdot \dyadsubsum Q E n A \mu(Q) = r_k \cdot \mu(A)^2.$$
The second term may be handled using a partitioning idea after changing the order of summation. Note that there is no contribution from $\ell = n$ by our convention that $\widetilde{Q}^{(n + 1)} = \emptyset$:
\begin{align*}
\dyadsubsum Q E n A \sum_{k \le \ell \le n} \mu(Q) \cdot \mu(\widetilde{Q}^{(\ell + 1)}) \cdot (r_{\ell + 1} - r_{\ell}) &= \sum_{k \le \ell \le n - 1} (r_{\ell + 1} - r_{\ell}) \dyadsubsum Q E n A \mu(Q) \cdot \mu(\widetilde{Q}^{(\ell + 1)}) \\
&= \sum_{k \le \ell \le n - 1} (r_{\ell + 1} - r_{\ell}) \sum_{\substack{P \in \dyad E {\ell + 1} \\ P \subseteq A}} \dyadsubsum Q E n P \mu(Q) \cdot \mu(P) \\
&= \sum_{k \le \ell \le n - 1} (r_{\ell + 1} - r_{\ell}) \dyadsubsum P E {\ell + 1} A \mu(P)^2.
\end{align*}
Finally, this yields that the overall change in $\mathcal{Q}$ from Case 1 is
$$\Delta_{\text{Case 1}} \mathcal{Q} = \left[1 - \left(\frac{\mu(A) + \mu(B)}{2\mu(A)}\right)^2\right] \cdot \left[r_k \cdot \mu(A)^2 + \sum_{k \le \ell \le n - 1} (r_{\ell + 1} - r_{\ell}) \dyadsubsum P E {\ell + 1} A \mu(P)^2\right].$$

\textbf{Case 2:} We may follow nearly identical reasoning to Case 1 with $B$ in the role of $A$. 
Following identical reasoning to Case 1, the overall change in $\mathcal{Q}$ from the $Q, R \subseteq B$ terms is
$$\Delta_{\text{Case 2}} \mathcal{Q} = \left[1 - \left(\frac{\mu(A) + \mu(B)}{2\mu(B)}\right)^2\right] \cdot \left[r_k \cdot \mu(B)^2 + \sum_{k \le \ell \le n - 1} (r_{\ell + 1} - r_{\ell}) \dyadsubsum P E {\ell + 1} B \mu(P)^2\right].$$

\textbf{Case 3:} Consider first the case that $Q \subseteq A$ and $R \subseteq B$. By the mass-exchange process, we have
\begin{align*}
\mu(Q) \mu(R) - \tilde \mu(Q) \tilde \mu(R) &= \mu(Q) \mu(R) - \left(\frac{\mu(A) + \mu(B)}{2 \mu(A)} \mu(Q)\right) \cdot \left(\frac{\mu(A) + \mu(B)}{2 \mu(B)} \mu(R)\right) \\
&= \left(1 - \frac{(\mu(A) + \mu(B))^2}{4 \mu(A) \mu(B)}\right) \mu(Q) \mu(R) \\
&= \left(\frac{4 \mu(A) \mu(B) - \mu(A)^2 - 2\mu(A) \mu(B) - \mu(B)^2}{4 \mu(A) \mu(B)}\right) \mu(Q) \mu(R) \\
&= \left(\frac{2 \mu(A) \mu(B) - \mu(A)^2 - \mu(B)^2}{4 \mu(A) \mu(B)}\right) \mu(Q) \mu(R) \\
&= - \frac{(\mu(A) - \mu(B))^2}{4 \mu(A) \mu(B)} \mu(Q) \mu(R).
\end{align*}
The corresponding change in $\mathcal{Q}$ is therefore
$$- \frac{(\mu(A) - \mu(B))^2}{4 \mu(A) \mu(B)} \dyadsubsum Q E n A \dyadsubsum R E n B \mu(Q) \mu(R) r(Q, R).$$
The symmetrized version yields the exact same result with $Q$ and $R$ interchanged, so the overall change in $\mathcal{Q}$ for this case is
$$\Delta_{\text{Case 3}} \mathcal{Q} = - \frac{(\mu(A) - \mu(B))^2}{2 \mu(A) \mu(B)} \dyadsubsum Q E n A \dyadsubsum R E n B \mu(Q) \mu(R) r(Q, R).$$
Now observe that the configuration of $Q \subseteq A$ and $R \subseteq B$ with $A \ne B$ but $\tilde{A} = \tilde{B}$ yields that $r(Q, R) = r_k$. Therefore,
\begin{align*}
\Delta_{\text{Case 3}} \mathcal{Q}  &= - \frac{(\mu(A) - \mu(B))^2}{2 \mu(A) \mu(B)} \cdot r_{k} \dyadsubsum Q E n A \dyadsubsum R E n B \mu(Q) \mu(R) \\
&= - \frac{(\mu(A) - \mu(B))^2}{2 \mu(A) \mu(B)}\cdot r_{k} \mu(A) \mu(B) \\
&= - \frac{1}{2} r_k(\mu(A) - \mu(B))^2.
\end{align*}

\textbf{Case 4:} First assume that $Q \subseteq A$ so that $R \nsubseteq (A \cup B)$. In this case,
$$\left[\mu(Q) \mu(R) - \tilde \mu(Q) \tilde \mu(R)\right] r(Q, R) = \left[1 - \left(\frac{\mu(A) + \mu(B)}{2 \mu(A)}\right)\right] \mu(Q) \mu(R) r(Q, R).$$
Note that because $Q \subseteq A$ and $R \nsubseteq A$, the function $Q \mapsto r(Q, R)$ is constant. Fix any $S \in \dyad E n$ which is not contained in $A \cup B$, and suppose that $\ell$ is the index for which $r(Q, S) = r_{\ell}$ for all $Q \subseteq A$ (with $1 \le \ell \le k$). We find that
\begin{align*}
\dyadsubsum Q E n A \left[\mu(Q) \mu(S) - \tilde{\mu}(Q) \tilde{\mu}(S)\right] r(Q, S) &= \left[1 - \left(\frac{\mu(A) + \mu(B)}{2 \mu(A)}\right)\right] \dyadsubsum Q E n A \mu(Q) \mu(S) r_{\ell} \\
&= \left[1 - \left(\frac{\mu(A) + \mu(B)}{2 \mu(A)}\right)\right] \mu(A) \mu(S) r_{\ell} \\
&= \left[\frac{\mu(A) - \mu(B)}{2\mu(A)}\right] \mu(A) \mu(S) r_{\ell} \\
&= \frac 1 2 (\mu(A) - \mu(B)) \mu(S) r_{\ell}.
\end{align*}
The symmetric terms, corresponding to $R \subseteq A$ and $Q \nsubseteq (A \cup B)$ are handled identically and yields the same result. This leads to
$$(\mu(A) - \mu(B)) \mu(S) r_{\ell}.$$
Summing over $\ell$ and $S$ (at the appropriate scale determined by $\ell$) yields the overall term, but we will delay this until the computation in Case 5.

\textbf{Case 5:} We proceed as in Case 4, by first assuming that $Q \subseteq B$ and $R \nsubseteq (A \cup B)$. Observe that
$$\left[\mu(Q) \mu(R) - \tilde \mu(Q) \tilde \mu(R)\right] r(Q, R) = \left[1 - \left(\frac{\mu(A) - \mu(B)}{2 \mu(B)}\right)\right] \mu(Q) \mu(R) r(Q, R).$$ 
As before, fix any $S \in \dyad E n$ with $S \nsubseteq A \cup B$; we find
\begin{align*}
\dyadsubsum Q E n B \left[\mu(Q) \mu(S) - \tilde \mu(Q) \tilde \mu(S)\right] r(Q, S) = \frac 1 2 (\mu(B) - \mu(A)) \dyadsubsum Q E n B \mu(Q) \mu(S) r_{\ell} 
= \frac 1 2 (\mu(B) - \mu(A)) \mu(S) r_{\ell}.
\end{align*}
The symmetric term with $Q\subsetneq (A \cup B)$ and $R \subseteq B$ contributes the same, yielding
$$(\mu(B) - \mu(A)) \mu(S) r_{\ell}.$$
This is precisely the negative of the corresponding term in Case 4, and the summands cancel. Thus the overall contribution from Cases 4 and 5 is zero.

\textbf{Case 6:} Here the mass-exchanged measure is completely unmodified, in that $\tilde \mu(Q) = \tilde \mu(R)$; this contribution is therefore zero.

We are now ready to combine the contributions from all six cases. 
\begin{align*}
\mathcal Q(\mu) - \mathcal Q(\tilde \mu) &= \left[1 - \left(\frac{\mu(A) + \mu(B)}{2 \mu(A)}\right)^2\right] \cdot \left[ r_k \cdot \mu(A)^2 + \sum_{k \le \ell \le n - 1} (r_{\ell + 1} - r_{\ell}) \dyadsubsum P E {\ell + 1} A \mu(P)^2 \right] \\
&\hspace{.2in} + \left[1 - \left(\frac{\mu(A) + \mu(B)}{2 \mu(B)}\right)^2\right] \cdot \left[ r_k \cdot \mu(B)^2 + \sum_{k \le \ell \le n - 1} (r_{\ell + 1} - r_{\ell}) \dyadsubsum P E {\ell + 1} B \mu(P)^2 \right] \\
&\hspace{.2in} - \frac 1 2 r_k \cdot (\mu(A) - \mu(B))^2 \\
&= \left[\mu(A)^2 + \mu(B)^2 - \frac{(\mu(A) - \mu(B))^2}{2} - 2 \left(\frac{\mu(A) + \mu(B)}{2}\right)^2\right] \cdot r_k \\
&\hspace{.2in} + \left[1 - \left(\frac{\mu(A) + \mu(B)}{2 \mu(A)}\right)^2\right] \sum_{k \le \ell \le n - 1} (r_{\ell + 1} - r_{\ell}) \dyadsubsum P E {\ell + 1} A \mu(P)^2 \\
&\hspace{.2in}  + \left[1 - \left(\frac{\mu(A) + \mu(B)}{2 \mu(B)}\right)^2\right] \sum_{k \le \ell \le n - 1} (r_{\ell + 1} - r_{\ell}) \dyadsubsum P E {\ell + 1} B \mu(P)^2.
\end{align*}
The first term vanishes, after observing that
$$a^2 + b^2 - \frac 1 2 (a - b)^2 - \frac 1 2 (a + b)^2 = 0$$
for all $a$ and $b$. This leaves us with the overall result.
\end{proof}

We may use the mass-exchange lemma to build a sequence of $\mathcal{Q}$-minimizing measures supported on the generations of a set $E$, and show that the $\mathcal{Q}$-minimizing measures have a particular equidistribution property. In the case that $k = n - 1$ (i.e. that $A$ and $B$ are the parents of two of the constituents of $E$), the lemma reduces to
\begin{align*}
\mathcal{Q}(\mu) - \mathcal{Q}(\tilde \mu) &= (r_n - r_{n - 1}) \left[\left[1 - \left(\frac{\mu(A) + \mu(B)}{2 \mu(A)}\right)^2\right] \mu(A)^2 + \left[1 - \left(\frac{\mu(A) + \mu(B)}{2 \mu(B)}\right)^2\right] \mu(B)^2\right] \\
&= (r_n - r_{n - 1}) \left[\mu(A)^2 + \mu(B)^2 - 2 \left(\frac{\mu(A) + \mu(B)}{2}\right)^2\right] \\
&= \frac 1 2 (r_n - r_{n - 1}) \cdot (\mu(A) - \mu(B))^2.
\end{align*}
Since $E$ has the filtration property, $r_n > r_{n - 1}$; this means that the $\mathcal{Q}$-minimizing measure supported on $E$ must assign equal mass to each sibling with a common parent at scale $n - 1$, or else a local mass-exchange would further decrease $\mathcal{Q}$. This is the desired local equidistribution property.

\subsection{Bootstrapping to global equidistribution} \label{RepulsionStep3}

The final result of the previous section applies only at the finest scale: if $A$ and $B$ are siblings at generation $n$ and $\mu$ is (the) repulsion-minimizing measure, then $\mu(A) = \mu(B)$. The ideal result would be to bootstrap this to all other generations: if $A$ and $B$ are constituents of $E$ at any localization, $\mu(A) = \mu(B)$. There is an obstacle to this task, however: if $A$ contains far more children than $B$ does, we should not expect equidistribution of the measure. In fact, we would want to bias the measure \emph{against} a region of $E$ which has many children, because there are far more pair interactions at finer scales. As such, we restrict our attention to a special case: we assume that different regions of the set have the same growth rate in terms of the number of components at generation $n$. To this end, recall that if $E$ is a set with the filtration property through generation $n$, we say that $E$ has the \emph{socialist filtration property} if there is a sequence of nonnegative integers $N_{\ell}$ such that for every $P \in \dyad E \ell$,
$$\left|\left\{Q \in \dyad E {\ell + 1} : Q \subseteq P\right\}\right| = N_{\ell}.$$
In other words: the filtration property means that every component at generation $\ell$ has a unique ancestor at generation $\ell - 1$, whereas the \emph{socialist} filtration property requires that each component at generation $\ell$ has exactly $N_{\ell}$ children at generation $\ell + 1$. Note that $N_{\ell}$ is assumed to be a nonnegative integer which can be zero; when $N_{\ell} = 0$, the construction terminates at that scale.


\begin{lemma} Fix a set $E$ and a natural number $n$ for which $E$ has the socialist filtration property through generation $n$. Let $\mu$ be a positive probability measure on $E$ which minimizes the repulsion $\mathcal{Q}$ induced by an increasing sequence $\{r_{\ell} : 0 \le \ell \le n\}$. If $A$ and $B$ are siblings at generation $k$, meaning that $A, B \in \dyad E {k + 1}$ with $\tilde{A} = \tilde B$ but $A \ne B$, then $\mu(A) = \mu(B)$.
\end{lemma}

\begin{proof}
We proceed recursively, working from the finest level ($\dyad E n$) to the roughest level ($\dyad E k$). To this end, suppose that $\mu$ is a repulsion-minimizing measure for the sequence $\{r_{\ell} : 0 \le \ell \le n\}$. At the finest level, choose $A, B \in \dyad E n$ with $A \ne B$ but $\widetilde A = \widetilde B$. As we noted above, if $\tilde{\mu}$ is the mass-exchanged measure corresponding to $\mu$, $A$, and $B$, the mass-exchange lemma implies
\begin{align*}
\mathcal{Q}(\mu) - \mathcal{Q}(\tilde \mu) = \frac 1 2 (r_n - r_{n - 1}) \cdot (\mu(A) - \mu(B))^2.
\end{align*}
Since $r_n > r_{n - 1}$ and $\mu$ minimizes the repulsion $\mathcal{Q}$, we must have $\mathcal{Q}(\mu) - \mathcal{Q}(\tilde \mu) \le 0$. This is only possible if $\mu(A) = \mu(B)$, establishing the most local equidistribution property.

Now consider $A, B \in \dyad E {n - 1}$ with $A \ne B$ but $\widetilde A = \widetilde B$. The mass-exchange lemma now yields two terms when moving from $\mu$ to $\tilde \mu$, corresponding to the two scales of the computation:
\begin{itemize}
\item $\ell = n - 2$: the contribution to $\mathcal{Q}(\mu) - \mathcal{Q}(\tilde \mu)$ is
$$(r_{n - 1} - r_{n - 2}) \left[\left(1 - \left(\frac{\mu(A) + \mu(B)}{2 \mu(A)}\right)^2\right) \dyadsubsum P E {n - 1} A \mu(P)^2 + \left(1 - \left(\frac{\mu(A) + \mu(B)}{2 \mu(B)}\right)\right)^2 \dyadsubsum P E {n - 1} B \mu(P)^2\right].$$
Because $A \in \dyad E {n - 1}$, the sum degenerates: there is only one $P$, and it is equal to $A$ itself. The second term is similarly degenerate, and the computation now proceeds exactly as in the base case. The overall contribution to $\mathcal{Q}(\mu) - \mathcal{Q}(\tilde \mu)$ is
$$\frac 1 2 (r_{n - 1} - r_{n - 2}) (\mu(A) - \mu(B))^2.$$
\item $\ell = n - 1$: the contribution to $\mathcal{Q}(\mu) - \mathcal{Q}(\tilde \mu)$ is
$$(r_{n} - r_{n - 1}) \left[\left(1 - \left(\frac{\mu(A) + \mu(B)}{2 \mu(A)}\right)^2\right) \dyadsubsum P E {n} A \mu(P)^2 + \left(1 - \left(\frac{\mu(A) + \mu(B)}{2 \mu(B)}\right)\right)^2 \dyadsubsum P E {n} B \mu(P)^2\right].$$
Consider the sum $\displaystyle \dyadsubsum P E n A \mu(P)^2$. Since we have already established equidistribution at this scale, we have $\mu(P) = \frac{1}{N_{\ell}} \mu(A)$ and thus
$$\dyadsubsum P E n A \mu(P)^2 = \left(\frac {\mu(A)} {N_{\ell}}\right)^2 \cdot N_{\ell} = \frac 1 {N_{\ell}} \mu(A)^2.$$
The second term, involving $B$, is handled identically. The overall contribution to $\mathcal{Q}(\mu) - \mathcal{Q}(\tilde \mu)$ is
$$\frac{1}{2 N_{\ell}} (r_n - r_{n - 1}) (\mu(A) - \mu(B))^2.$$
\end{itemize}
Combining these two terms yields
$$\mathcal{Q}(\mu) - \mathcal{Q}(\tilde \mu) = \frac 1 2 \left(\mu(A) - \mu(B)\right)^2 \cdot \left[(r_{n - 1} - r_{n - 2}) + \frac 1 {N_{\ell}} (r_n - r_{n - 1})\right].$$
As before: since $\mu$ is a repulsion minimzer, the left hand side must be at most zero and we have $\mu(A) = \mu(B)$. This is the desired equidistribution at the second finest scale.

Now consider the general case: suppose that $A, B \in \dyad E {\ell}$ for some $\ell \le n - 2$ and that $\mu$ is equidistributed on all finer scales. We split the contributions to $\mathcal{Q}(\mu) - \mathcal{Q}(\tilde \mu)$ into each scale $j \ge \ell$:
\begin{itemize}
\item If $j = \ell$,
$$\dyadsubsum P E {\ell} A \mu(P)^2 = \mu(A)^2$$
because the only summand corresponds to $P = A$. The sum with $B$ is handled similarly and the overall contribution to $\mathcal{Q}(\mu) - \mathcal{Q}(\tilde \mu)$ is $\frac 1 2 (r_{\ell + 1} - r_{\ell}) (\mu(A) - \mu(B)^2)$ and is strictly positive unless $\mu(A) = \mu(B)$.
\item If $j > \ell$, we need to calculate $\displaystyle \dyadsubsum P E {j} A \mu(P)^2$. Since we have already established equidistribution, the sum will reduce to $\frac 1 M_ \mu(A)^2$ where $M_j$ is the number of components at scale $j$ contained within $A$. To be explicit,
$$M_j = \prod_{\ell \le i < j} N_i.$$
The most important point here is that the corresponding sum over subsets of $B$ picks up the exact same factor $M_j$. Overall, the contribution to $\mathcal{Q}(\mu) - \mathcal{Q}(\tilde \mu)$ at scale $j$ is thus
$$\frac 1 {2 M_j} (r_j - r_{j - 1}) (\mu(A) - \mu(B))^2.$$
As before: this is strictly positive unless $\mu(A) = \mu(B)$.
\end{itemize}
Combining the previous two analyses, we have (note that $M_{\ell} = 1$ by the standard convention on the empty product) that
$$\mathcal{Q}(\mu) - \mathcal{Q}(\tilde \mu) = \frac 1 2 (\mu(A) - \mu(B))^2 \cdot \sum_{\ell \le j \le n} \frac{r_j - r_{j - 1}}{M_j}.$$
If $\mathcal{Q}(\mu)$ is minimal, the left hand side is at most zero; the right hand side is at least zero since $r_j > r_{j - 1}$ for all $j$. The only way these observations are compatible is if $\mu(A) = \mu(B)$, which is the desired result.
\end{proof}

In order to explain why the idea of the previous lemma works, consider that when we exchange mass between two regions, there is the effect at the primary scale ($\widetilde A = \widetilde B \in \dyad E {k}$) as well as an effect from any finer scale. Ordinarily, the finer scales would be an obstacle to the calculation; however, we already have established equidistribution and the sums (which are \emph{a priori} quite complicated) are in fact susceptible to the sharpest version of the Cauchy-Schwarz inequality.

We are now ready to bootstrap this to a global result: if $A$ and $B$ are any two components in $\dyad E k$, we have $\mu(A) = \mu(B)$ for a repulsion-minimizing measure. To see this, consider a common ancestor $P = \widetilde{A}^{(j)} = \widetilde{B}^{(j)}$ which contains both descendants. The previous lemma implies that $\mu(\widetilde A^{(j + 1)}) = \mu(\widetilde B^{(j + 1)})$. We can then use the equidistribution within $A^{(j + 1)}$ to conclude that
$$\mu(\widetilde{A}^{(j + 2)}) = \frac 1 {N_{j + 1}} \mu(\widetilde A^{(j + 1)})$$
and likewise for $\widetilde B^{(j + 2)}$. Carrying out this computation at each scale naturally leads to
$$\mu(A) = \left(\prod_{\ell = j + 1}^{k - 1} \frac 1 {N_{\ell}} \right) \mu(\widetilde A^{(j + 1)}) = \left(\prod_{\ell = j + 1}^{k - 1} \frac 1 {N_{\ell}} \right) \mu(\widetilde B^{(j + 1)}) = \mu(B).$$
Applying this at the final stage of the construction, we have proven the following theorem:
\begin{theorem} \label{thm:socialist}
Fix a set $E$ and a natural number $n$ for which $E$ has the socialist filtration property through generation $n$, as well as a sequence $\{r_{\ell} : 0 \le \ell \le n\}$. Let $\mu$ be a positive Borel probability measure supported on $E$ which minimizes the repulsion induced by $\mathcal{Q}$. If $A, B \in \dyad E n$, we have $\mu(A) = \mu(B).$
\end{theorem}

\section{Applications to Riesz 1-energy}

So far, the results of this section involve repulsions induced by any increasing sequence $\{r_{\ell} : 0 \le \ell \le n\}$. When the repulsions $r_{\ell}$ are determined in a way that reflects the geometry of the how a set $E$ is distributed geometrically, it is possible to directly connect $\mathcal{Q}(\mu)$ to the Riesz 1-energy $I_1(\mu)$.  The core idea is that we have a decomposition
$$I_1(\mu) = \iint \frac{d\mu(x) \, d\mu(y)}{|x - y|} = \dyadsum Q E n \dyadsum R E n \int_Q \int_R \frac{d\mu(x) \, d\mu(y)}{|x - y|}.$$
If $Q$ and $R$ are constituents which are separated from each other (to be precise, the distance between $Q$ and $R$ should be at least a fraction of the diameter of $Q$ and $R$), the integral is comparable to $\mu(Q) \mu(R) \operatorname{dist}(Q, R)$. This suggests two useful properties: 
\begin{itemize}
\item At each generation $k$, every constituent in $\dyad E k$ has roughly the same diameter.
\item For each constituent $A \in \dyad E k$, its children in $\dyad E {k + 1}$ should be separated on a scale comparable to the diameter of $A$.
\end{itemize}
In this case, we choose the sequence of repulsions so that $r_k \sim \operatorname{diam}(A)^{-1}$, where $A \in \dyad E k$. 

\begin{defn}
Fix a natural number $n$ and a set $E$ with the filtration property through generation $n$. Suppose that there are positive constants $C$ and $\epsilon$ and an increasing sequence $\{r_k : 0 \le k \le n\}$ such that for each $k$ we have:
\begin{itemize}
\item for each $A \in \dyad E k$, 
$$\frac 1 C r_k \le \operatorname{diam}(A)^{-1} \le C r_k,$$
\item and if $A, B \in \dyad E k$ with $\widetilde{A} = \widetilde{B}$ but $A \ne B$, 
$$\operatorname{dist}(A, B) \ge \epsilon r_k^{-1}.$$
\end{itemize}
We then say that $E$ is \emph{evenly distributed} with respect to the sequence $\{r_k\}$.
\end{defn}
 
\begin{theorem}\label{thm:energy-bounded-repulsion}
Fix a natural number $n$ and a set $E$ with the socialist filtration property through generation $n$. Suppose that $E$ is evenly distributed with respect to the sequence $\vec r = \{r_k\}$ and that $\mathcal{Q}_{\vec r}$ is the corresponding repulsion. Then
$$I_1(\mu) \gtrsim \mathcal{Q}_{\vec r}(\mu).$$
\end{theorem}

\begin{proof}
As indicated above, we write
$$I_1(\mu) = \dyadsum Q E n \dyadsum R E n \int_Q \int_R \frac{d\mu(x) \, d\mu(y)}{|x - y|}.$$
If $Q \ne R$, the integrand is comparable to the distance between $Q$ and $R$, which is in turn comparable to the diameter of their last common ancestor. If $Q = R$, the integrand is at least $\operatorname{diam}(Q)^{-1}$, which is comparable to $r_n$ by assumption. The result follows.
\end{proof}

The prototypical examples which have the desired geometric properties are approximations to fractal $1$-sets, such as the $n$-th generation $\mathcal{K}_n$ of the four-corner Cantor set. We have the following corollary:

\begin{corollary}  \label{cor:CantorEx}
Fix a natural number $n$ and consider the repulsion $\mathcal{Q}$ generated by the sequence $\{4^{\ell} : 0 \le \ell \le n\}.$ If $\mu$ is a positive probability measure supported on $\mathcal{K}_n$,
$$I_1(\mu) \gtrsim n.$$
\end{corollary}

\begin{proof}
It is immediate that $E:= \mathcal{K}_n$ has the socialist filtration property through generation $n$, choosing $N_{\ell} = 4$ for $\ell < n$ and $N_n = 0$. Furthermore, $E$ is evenly distributed with respect to the sequence $\{4^{\ell} : 0 \le \ell \le n\}$. Applying Theorem \ref{thm:socialist}, the repulsion-minimizing measure on $\mathcal{K}_n$ has the equidistribution property: if $Q$ is a constituent cube at generation $n$, $\mu(Q) = 4^{-n}$. 

We are also able to directly count the number of cubes corresponding to each potential value of $r(Q, R)$: given $Q \in \dyad E n$ and a generation $\ell \le n$, there are exactly $3 \cdot 4^{n - \ell - 1}$ cubes $R$ for which
$$\widetilde Q^{(\ell)} = \widetilde R^{(\ell)} \quad \text{ but } \quad \widetilde Q^{(\ell + 1)} \ne \widetilde R^{(\ell + 1)}.$$
Specifically, these are the generation $n$ descendants of three siblings of $\widetilde{Q}^{(\ell + 1)}$ contained in $\widetilde Q^{(\ell)}$. These cubes are also characterized as the cubes for which
$$r(Q, R) = 4^{\ell}.$$

Neglecting the diagonal contribution, we can then calculate
\[
\mathcal{Q}(\mu) = \dyadsum Q E n \dyadsum R E n \mu(Q) \mu(R) r(Q, R) 
\ge \dyadsum Q E n \sum_{\ell = 0}^{n - 1} \sum_{\substack{R \in \dyad E n \\ \widetilde{Q}^{(\ell)} = \widetilde{R}^{(\ell)} \\ \widetilde{Q}^{(\ell + 1)} \ne \widetilde{R}^{(\ell + 1)}}} \mu(Q) \mu(R) r(Q, R).
\]
For the innermost sum, note that $r(Q, R) = 4^{\ell}$ is constant in this configuration. Moreover, the sum of measures $\mu(R)$ is equal to
$$\sum_{\substack{R \in \dyad E n \\ \widetilde{Q}^{(\ell)} = \widetilde{R}^{(\ell)} \\ \widetilde{Q}^{(\ell + 1)} \ne \widetilde{R}^{(\ell + 1)}}} \mu(R) = \mu(\widetilde{Q}^{(\ell)} \setminus \widetilde{Q}^{(\ell + 1)}) = \frac 3 4 \mu(\widetilde{Q}^{(\ell)}).$$
Combining these results,
$$\mathcal{Q}(\mu) \ge \sum_{\ell = 0}^{n - 1} \dyadsum Q E n \mu(Q) \cdot \frac 3 4 \mu(\widetilde{Q}^{(\ell)}) \cdot 4^{\ell}.$$
Finally, using that $\mu(\widetilde Q^{(\ell)}) = 4^{-\ell}$ and $\mu$ is a probability measure, by Theorem \ref{thm:energy-bounded-repulsion},
\[I_1(\mu) \gtrsim \mathcal{Q}(\mu) \gtrsim \sum_{\ell = 0}^{n - 1} 1 = n. \qedhere\]
\end{proof}

\begin{corollary}
For all natural numbers $n$, the Riesz 1-capacity of $\mathcal{K}_n$ is comparable to $1/n$.
\end{corollary}

\begin{proof}
By the previous lemma, $\operatorname{Cap}_1(\mathcal{K}_n) \gtrsim 1/n$ since any positive Borel probability measure supported on $\mathcal{K}_n$ has $1$-energy $\gtrsim n$. This result is sharp by considering the equidistributed measure 
\[
\mu = \frac{\left.\mathcal{L}^2\right|_{\mathcal{K}_n}}{\mathcal{L}^2(\mathcal{K}_n)}. \qedhere
\]
\end{proof}

\section{The general setting: connections with quadratic forms}

Consider the general setting where a set $E$ is given as a disjoint union
$$E = \bigcup_{i = 1}^n B(z_i, r)$$
for a fixed scale $r > 0$ and a set of points $\{z_i : 1 \le i \le n\}$ which are separated on scale $r$:
$$|z_i - z_j| \ge 2r \quad \text{ for } i \ne j.$$
Such a set does not necessarily have any filtration property available, and the minimum $1$-energy of a positive probability measure supported on $E$ depends on the precise geometric distribution of points. This distribution can be encoded by the following matrix.

\begin{defn}
Given a scale $r > 0$ and a collection of points $\{z_i : 1 \le i \le n\}$ with $|z_i - z_j| \ge 2r$ for all $i \ne j$, the corresponding \emph{repulsion matrix} is defined by
$$\mathbb{A}_{i, j} = (r + |z_i - z_j|)^{-1}.$$
\end{defn}
Note that, by the separation condition, we have that
$$\mathbb{A}_{i, j} \sim \left\{\begin{array}{lr} r^{-1} & i = j \\
|z_i - z_j|^{-1} & i \ne j\end{array}.\right.$$
The repulsion matrix has a number of desirable properties -- namely, it is symmetric and positive definite. The first property is immediate, while the second follows from Schoenberg's interpolation theorem; see, e.g. \cite[Chapters 14 and 15]{LigWar09}. In particular, the matrix is invertible.

We have the following comparison between the $1$-energy of a measure and a property of the repulsion matrix; this is strongly related to \cite{RajRanRos10}. For notation, we let $\vec 1$ be the (column) vector of all ones in $\mathbb{R}^n$.

\begin{theorem}
Fix a scale $r > 0$ and a collection of points $\{z_i : 1 \le i \le n\}$ with $|z_i - z_j| \ge 2r$ for all $i \ne j$. If $\mu$ is a positive Borel probability measure supported on $E$ and $A$ is the corresponding repulsion matrix,
$$I_1(\mu) \gtrsim (\vec 1\hspace{.05cm}{}^T  \mathbb{A}^{-1} \vec 1)^{-1}.$$
As a consequence, the Riesz 1-capacity of $E$ has the estimate
$$\operatorname{Cap}_1(E) \gtrsim \vec 1\hspace{.05cm}{}^T \mathbb{A}^{-1} \vec 1.$$
\end{theorem}

\begin{proof}
If $\mu$ is a positive Borel probability measure supported on $E$, form a vector $\vec \mu$ whose $i$-th entry is the $\mu$-measure of $B(z_i, r)$; this is a vector whose entries are nonnegative and sum to one. As in the case where the filtration property is available, we can write the $1$-energy in terms of a quadratic form:
$$I_1(\mu) = \iint  \frac{d\mu(x) \, d\mu(y)}{|x - y|} = \sum_{i = 1}^n \sum_{j = 1}^n \int_{B(z_i, r)} \int_{B(z_j, r)} \frac{d\mu(x) \, d\mu(y)}{|x - y|}.$$
In the diagonal case, $|x - y|^{-1} \ge r^{-1}$; in the off-diagonal case, $|x - y|^{-1} \gtrsim |z_i - z_j|$. Therefore,
$$I_1(\mu) \gtrsim \sum_{i = 1}^n \sum_{j = 1}^n \mu(B(z_i, r)) \mu(B(z_j, r)) \mathbb{A}_{i, j} = \vec \mu^{T} \mathbb{A} \vec \mu.$$
We now wish to minimize this over measures $\mu$, in order to get a uniform lower bound for the $1$-energies. 
Now consider the optimization problem
$$(\ast) \quad \quad \operatorname{min} \left\{\vec x^T \mathbb{A} \vec x : \vec x^T \vec 1 = 1\right\}.$$
This can be solved using a variational method (or by Lagrange multipliers). Since $\mathbb{A}$ is positive definite, a unique minimizer (subject to the affine constraint $\vec x^T \vec 1 = 1$) must exist. Denote it as $x^{\ast}$; perturb it via a vector $\vec \eta$ with $\vec \eta^T \vec 1 = 0$. Since $x^{\ast}$ is the unique minimizer, 
\begin{align*}
0 &= \left.\frac{d}{dt}\right|_{t = 0} (x^{\ast} + t \vec \eta)^T \mathbb{A} (x^{\ast} + t \vec \eta) \\
&= \left.\frac{d}{dt}\right|_{t = 0} \left({x^{\ast}}^T \mathbb{A} x^{\ast} + 2t \vec \eta \mathbb{A} x^{\ast} + t^2 \vec \eta^T \mathbb{A} \vec \eta\right) \\
&= 2 \vec \eta \mathbb{A} x^{\ast}.
\end{align*}
Since this holds for all $\vec \eta$ which are orthogonal to $\vec 1$, it follows that $\mathbb{A} x^{\ast} = \lambda \vec 1$ for a scalar $\lambda$; the minimum value of the quadratic form is then
$$x^{\ast} \mathbb{A} x^{\ast} = \lambda x^{\ast} \vec 1 = \lambda.$$
The scalar may be found through the normalization condition for $x^{\ast}$: since $x^{\ast} = \lambda \mathbb{A}^{-1} \vec 1$, we have
$$1 = {x^{\ast}}^T \vec 1 = (\lambda \mathbb{A}^{-1} \vec 1)^T \vec 1 = \lambda (\vec 1\hspace{.05cm}{}^T  \mathbb{A}^{-1} \vec 1).$$
This can be rearranged to yield the desired result.
\end{proof}

The quantity $\vec 1\hspace{.05cm}{}^T  \mathbb{A}^{-1} \vec 1$ is the sum of the entries of the inverse matrix $\mathbb{A}^{-1}$, although this does not have an immediately apparent geometric description. Under two additional hypotheses, however, there is a far simpler description. We will assume that each column of the matrix $\mathbb{A}$ has a (roughly) common sum and that the minimizer $x^{\ast}$ is nonnegative.

\begin{theorem}
Fix a scale $r > 0$ and a collection of points $\{z_i : 1 \le i \le n\}$ with $|z_i - z_j| \ge 2r$ for all $i \ne j$. Let $\mathbb{A}$ be the corresponding repulsion matrix and $x^{\ast}$ the unique minimizer for $(\ast)$. Assume further that each row-sum $\sum_j \mathbb{A}_{i, j}$ is comparable to a quantity $S$ and that each entry of $x^{\ast}$ is nonnegative. If $\mu$ is a positive Borel probability measure supported on $E$,
$$I_1(\mu) \gtrsim S^{-1}.$$
\end{theorem}

\begin{proof}
Recall that the minimizer $x^{\ast}$ satisfies the equation $A x^{\ast} = \lambda \vec 1.$ Writing out the matrix multiplication, we have
$$\sum_j \mathbb{A}_{i, j} x^{\ast}_j = \lambda$$
for each $1 \le i \le n.$ Sum this in $i$ and reverse the order:
\begin{align*}
\lambda n &= \sum_{i = 1}^n \lambda = \sum_{i = 1}^n \sum_{j = 1}^n \mathbb{A}_{i, j} x^{\ast}_j = \sum_{j = 1}^n \sum_{i = 1}^n \mathbb{A}_{i, j} x^{\ast}_j \sim \sum_{j = 1}^n S x^{\ast}_j = S.
\end{align*}
Therefore, $\lambda \sim S / n$.
\end{proof}

As an application of the previous result, consider again the $n$-th generation $\mathcal{K}_n$ of the four-corner Cantor set. Using the square-counting argument of \cite{BatVol10}, each row-sum of the corresponding repulsion matrix is comparable to $n \cdot 4^n$ and the number of constituents is $4^n$. Thus we are left with
$$I_1(\mu) \gtrsim \frac 1 n.$$
This is consistent with the results of Corollary \ref{cor:CantorEx}. Numerical results (through $\mathcal{K}_7$) indicate that the minimizer $x^{\ast}$ is always nonnegative; however, the author is aware of a proof for all generations. Numerical results also strongly suggest this for an arbitrary point configuration in the plane. As such, we state the following conjecture:
\begin{conjecture}
If $\mathbb{A}$ is the repulsion matrix of a set $E = \displaystyle \bigcup_{i = 1}^n B(z_i, r)$ with $|z_i - z_j| \ge 2r$ whenever $i \ne j$, then $\mathbb{A}^{-1} \vec 1 \ge \vec 0.$
\end{conjecture}

\section{Acknowledgments} 

Part of this work was completed at the American Institute of Mathematics during the SQuaRE ``Covering Fractals by Curves.'' The author is very grateful to her collaborators Paige Bright, Caleb Marshall, and Krystal Taylor; their conversations, suggestions, and support were critical in developing this paper. The author would also like to sincerely thank Jimmie Adriazola for suggesting a connection to the fast multipole method, without which this paper would not have been possible. The original inspiration for this project was a question from Ben Jaye regarding the usage of energy techniques for Favard length.
\bibliography{refs_energy}
\bibliographystyle{abbrv}

\end{document}